\documentclass[a4paper,11pt]{amsart}
\usepackage[T2A]{fontenc}
\usepackage[cp1251]{inputenc}
\usepackage[dvips]{graphicx}
\usepackage{amsmath}
\usepackage{amsfonts}
\usepackage{amssymb}
\usepackage{amsthm}
\usepackage{euscript}
\usepackage{xypic}
\author[Ievgen Makedonskyi]{Ievgen Makedonskyi}
\address{
Algebra Department, Faculty of Mechanics and Mathematics,
Kyiv Taras Shevchenko University\\
64, Volodymyrskaia street, 01033 Kyiv, Ukraine}
\email{makedonskyi@univ.kiev.ua}
\title[On wild Lie algebras]{On wild Lie algebras}

\DeclareMathOperator\id{id}

\DeclareMathOperator\differential{d}
\renewcommand\d\differential

\let\leq\leqslant
\let\geq\geqslant

\let\star *
\let\subset\subseteq

\newtheorem{theorem}{Theorem}
\newtheorem{lemma}{Lemma}
\newtheorem{proposition}{Proposition}
\newtheorem{corollary}{Corollary}
\newtheorem{definition}{Definition}
\theoremstyle{definition}

\newtheorem{remark}{Remark}

\begin{document}

\sloppy

\begin{abstract}
 We give a criterion of tameness and wildness for a finite-dimensional Lie algebra over an algebraically closed field.
\end{abstract}

\maketitle
\section{Introduction}
All considered Lie algebras are defined over a fixed
algebraically closed field $\mathbb{K}$ of zero characteristic
and are finite-dimensional over $\mathbb{K}$.
Throughout in the text we will not write the field.
In particular, by $sl_n$ we will denote $sl_n(\mathbb{K})$.
All definitions of tameness, wildness and controlled wildness correspond to such definitions from the work \cite{Rin}.

There exist many papers about the tameness and wildness (see for example
\cite{Gel}, \cite{Drozd}, \cite{Rin}).

Let $L$ be a finite-dimensional Lie algebra. We will consider the classification problem of
finite dimensional linear representations of $L$ up to equivalence.
This work is devoted to the following question: determine all Lie algebras $L$ such that
the problem of classification of all representations of $L$ is wild or tame.
 Such algebras we will call wild and tame respectively.
 The main result of this paper is:

  Theorem 2.

There exist only five classes of tame Lie algebras:

1) semisimple Lie algebras;

2) the one-dimensional Lie algebra;

3) the direct sums of semisimple and the one-dimensional Lie algebras;

4) $sl_2 \rightthreetimes I$, where $I$ is the two-dimensional indecomposable module;

5) the direct sums of semisimple Lie algebras and $sl_2 \rightthreetimes I$.

Any other Lie algebra is controlled wild.

So we have proved that the classification problems of the representations of the Lie algebras are either tame
or controlled wild.

\section{Examples of tame and wild Lie algebras}

\subsection{Two-dimensional Lie algebras are wild}
 Let $L$ be a two-dimensional Lie algebra. Then $L$ is either abelian or has a basis $L= \langle x,y\rangle$
such that $\left[ x,y\right]=y$. The classification problem for the representations of the two-dimensional abelian Lie algebra
coincides with the classification problem for two commuting matrices. This problem is controlled wild.
The classification problem for two matrices satisfying the relation
 $\left[ x,y\right]=y$ is controlled wild too (see for example \cite{Sam}). Thus,
any two-dimensional Lie algebra is controlled wild.

\subsection{Solvable Lie algebras are wild}
 Assume that $L \rhd I$, $\dim L / I=2$. Then $L$ is controlled wild because $L / I$
 is two-dimensional.

 In particular, let $L$ be a solvable Lie algebra.
 Assume that $\dim{L} \geq 2$. Then $L$ has an ideal of codimension $2$.
  Therefore, we have:
  \begin{proposition}\label{solv}
  Any solvable Lie algebra $L$ with $\dim L>1$ is controlled wild.
  \end{proposition}

  The following fact is obvious.
  \begin{remark} \label{summa}
  Assume $L=I \oplus J$ is a direct sum of Lie algebras, where $I$ is a wild Lie algebra. Then $L$ is a wild Lie algebra too.
  \end{remark}

\subsection{Semisimple Lie algebras are tame} On the other hand, the classical representation theory of semisimple Lie algebras
implies
 that all semisimple Lie algebras are tame.

\subsection{One-dimensional extensions of semisimple Lie algebras are tame}
The next proposition seems to be known, but having no precise reference we supply it with
a complete proof.
 \begin{proposition} \label{tame} \label{sumtame}
Let $L=\widehat{L}\oplus L_1$ be the Lie algebra such that $\widehat{L}$
is semisimple. Let $(M,f)$ be an irreducible representation of $L$.
 Then there exist indecomposable representations $(M_1,f_1)$ of the algebra $\widehat{L}$ and $(M_2,f_2)$ of the algebra
 $L_1$ such that $M=M_1\otimes M_2$,
 $f(X+Y)=f_1(X)\otimes \id + \id \otimes f_2(Y)$, $X \in \widehat{L}$,
 $Y \in L_1$ and $\id$ is the identity operator.
 \end{proposition}
 \begin{proof}
 The semisimplicity of $\widehat{L}$ implies that the representation
$\widehat{L} \ni X \mapsto f(X)$ is completely decomposable.
Therefore we can assume that
$f(X)=\left(
        \begin{array}{cccc}
          F_1(X) & 0 & \ldots & 0 \\
          0 &  F_2(X) & \ldots & 0 \\
          \vdots & \vdots & \ddots & \vdots \\
          0 & 0 & \ldots &  F_k(X) \\
        \end{array}
      \right),
$
where $F_i$ is the irreducible representation of the algebra $\widehat{L}$, the dimension of this representation will be denoted by $h_i$.
Next, we may assume, that the representations $F_1,...F_p$ are pairwise equivalent,
 and the representations $F_q, q>p$ are not equivalent to $F_1$.
 Let $S_{ij}$ be $h_i \times h_j$-matrices and $S=\left(
        \begin{array}{ccc}
          S_{11} &  \ldots & S_{1k} \\
          \vdots  & \ddots & \vdots \\
          S_{k1}  & \ldots &  S_{kk} \\
        \end{array}
      \right)\in gl(n, \mathbb{K})$.
 Assume that $f(X)S=Sf(X)$ for any $X \in \widehat{L}$. Then $f_i(X)S_{ij}=S_{ij}f_j(X)$,
 $i,j=1, \ldots k$.
 By the Shur's Lemma (\cite{Goto}, p. 225) we have $S_{ij}=s_{ij}1_{h_1}$, $s_{ij}\in \mathbb{K}$, $i,j=1, \ldots p$ and
 $S_{ij}=0$, $i \leq p <j$ У $j \leq p <i$.

 Let us apply this result to $S=f(Y)$, $Y \in L_1$. Assume that $p \neq k$.
 Then the relations $S_{ij}=0$, $i \leq p <j$ and $j \leq p <i$ imply that $(M,f)$ is decomposable,
a contradiction. Therefore, $p=k$.

Denote $f_1(X)=F_1(X)$ for $X \in \widehat{L}$, $f_2(Y)=(s_{ij})\in gl(p,\mathbb{K})$,
 $Y \in L_1$. It is obvious that $M$ is decomposable into a tensor product. Thus, we have:
 \[f(X)=f_1(X)\otimes 1_p, X \in \widehat{L};\]
 \[f(Y)=1_h \otimes f_2(Y), Y \in L_1.\]
 It is easy to see that if the representation $(M_2,f_2)$ is decomposable into the direct sum of two other representations, then the representation
$f$ is decomposable into the direct sum of two other representations, too. Therefore, $(M_2,f_2)$ is an indecomposable representation.
Conversely, if the representation $(M_2,f_2)$ is indecomposable, then the representation  $(M,f)$ is indecomposable, too.
 \end{proof}
 \begin{remark}\label{rtame}
 The direct sum of a tame and semisimple Lie algebras is tame.
  The representations of the one-dimensional Lie algebra $\langle  e \rangle$ are given by the image of the element $e\mapsto f_2(e)$.
 Therefore for a semisimple Lie algebra $\widehat{L}$ any indecomposable representation
 of the Lie algebra $L=\widehat{L}\oplus \langle e \rangle$,
 is given by the simple representation of the algebra $\widehat{L}$ and the Jordan cell. Therefore, all these algebras are tame.

 More generally, the direct sums of semisimple and tame Lie algebras are tame.
 \end{remark}

 \section{Quiver of Lie algebra with an abelian radical}
 Now we are going to reduce the study of representations of Lie algebras with abelian radical to the study of representations of certain quivers.

 \begin{lemma}
 Let $\widehat{L}=L \rightthreetimes I$ be a Lie algebra such that $L$ is semisimple, $I$ is an abelian ideal.
 Then the category of representations of the algebra  $\widehat{L}$ is equivalent to the following category:
 the objects of this category will be pairs $(M, \phi)$, where $M$ is an
 $L$-module, $\phi: I \otimes M \rightarrow M$ is the module homomorphism such that
 \begin{equation} \label{ex}
 \phi  \circ (\id \otimes \phi)
 \left(\left(I \wedge I \right) \otimes M\right)=0.
 \end{equation}
 The morphisms in this category are the commutative diagrams:

\begin{equation}\label{com}
 \xymatrix{
I\otimes M\ar[d]^{\id \otimes \alpha} \ar[r]_\phi &M\ar[d]_{ \alpha}\\
I\otimes N\ar[r]^\psi&N},
\end{equation}
 $\alpha$ is the module homomorphism.

 \end{lemma}
 \begin{proof}
 Let $M$ be $\widehat{L}$-module. Therefore $M$ is $L$-module. Let the map $\phi: I \otimes M \rightarrow M$ be given by
 $i \otimes m \mapsto i  m$. Then for any $i,j \in I, m \in M$:
 \[\phi  \circ (\id \otimes \phi)((i\otimes j - j\otimes i) \otimes(m))=i(jm)-j(im)=[i,j]m=0.\]
 Therefore the condition \eqref{ex} holds.

 Also we have:
 \[l \phi (i \otimes m)=lim=[l,i]m + ilm= \phi(l ( i\otimes m))\]
 Hence $\phi$ is the module morphism.

 Conversely, let $\phi$ be a module morphism $\phi: I \otimes M \rightarrow M$ satisfying the condition \eqref{ex}.
 Define the action of $i \in I$ on $m \in M$ by the rule $im:= \phi(i \otimes m)$. By the condition \eqref{ex}
 we get for $i,j \in I, m \in M$: $[i,j]m=ijm-jim=0$, besides that, for $l \in L, i \in I, m \in M$ we have:
 \[[l,i]m=\phi([l,i] \otimes m)=\phi(l(i \otimes m)-i \otimes lm)=l\phi(i \otimes m)-ilm=lim-ilm \].
 Therefore $M$ with the given action is an $\widehat{L}$-module.

 It is obvious that the two considered maps are inverse to each other. Thus we have the correspondence between the objects.

 Now let $A$ be a morphism of the $\widehat{L}$-modules $A: M\rightarrow N$.
 Then it is also a morphism of the $L$-modules. Denote this morphism by $\alpha$.
 Then for $i \in I, m \in M$ we obtain that
 \[\alpha(\phi(i \otimes m))=\alpha(im)=A(im)=iA(m)=i\alpha(m)=\psi(i \otimes \alpha(m)).\]
 Therefore this map gives the commutative diagram \eqref{com}.
 By the same calculation we have that any commutative diagram \eqref{com} gives the morphism of the
 $\widehat{L}$-modules. Thus we proved that these two categories are equivalent.
 \end{proof}

 \begin{remark}\label{com1}
 In the same way we have that if $\widehat{L}=L \rightthreetimes R$, $L$ is a semisimple Lie algebra and the radical $R$
 is generated by the submodule $I$, $I\backsimeq R/[R,R]$, then the category of representations of $\widehat{L}$ is equivalent
 to the category of pairs $(M, \phi)$, where $M$ is an
 $L$-module, $\phi: I \otimes M \rightarrow M$ is a module morphism with some set of relations,
 which hold if the relations \eqref{ex} hold (i. e. we have the inclusion of the ideals).
 \end{remark}

For a given Lie algebra $L \rightthreetimes I$, where $L$ is semisimple and  $I$ is an $L$-module,
let us introduce  the quiver $K_I$. The vertices of this quiver are
equivalence classes of irreducible $L$-modules, the number of arrows from $M$ to $N$ is equal to the multiplicity of $N$ in $I\otimes M$.

\begin{lemma}
The category of all finite dimensional representations of $K_I$ is equivalent to the category
$\mathcal{K}_I$ of pairs $(M, \phi: I \otimes M \rightarrow M)$,
with morphisms given by the commutative diagrams \ref{com}.
\end{lemma}
\begin{proof}
Decompose the module $M$ into the direct sum of irreducible components:
\[M=\bigoplus_{i=1}^n V_i \otimes M_i,\]
where all $M_i$ are pairwise nonequivalent irreducible components of $M$, $V_i$ are vector spaces and $\dim V_i$ is the multiplicity of  $M_i$ in $M$.
Furthermore:
\[I \otimes M=\bigoplus_{i=1}^n V_i \otimes I \otimes M_i=\bigoplus V_i \otimes M_{ij},\]

$I \otimes M_i=\bigoplus M_{ij}$ is the decomposition of $I \otimes M_i$ into the direct sum of the irreducible (possibly
 not pairwise nonequivalent) components. Hence the map $\phi: I \otimes M \rightarrow M$\
is decomposed into the direct sum of maps $V_i \otimes M_{ij} \rightarrow V_l \otimes M_l$,
this map can be nonzero only if $M_{ij}\cong M_l$. Therefore the map
$I \otimes M \rightarrow M$ is given by the set of the maps $\alpha({i,j}):V_i \rightarrow V_l$ for every
$M_{ij}\cong M_l$. Note that the representations of the quiver $K_I$ are given by the same way. Thus we have the correspondence
between the objects. The morphism from $M=\bigoplus_{i=1}^n V_i \otimes M_i$ to $N=\bigoplus_{i=1}^n U_i \otimes M_i$
(possibly, some of $U_i$ and $V_i$ are zero) is the set of the maps
$\gamma_i=V_i \rightarrow U_i$, the commutativity of the diagram \eqref{com} is obviously equivalent to
the commutativity of the following diagrams:
\begin{equation}
 \xymatrix{
V_{i}\ar[d]^{ \gamma_i} \ar[r]_{\alpha_{ij}} &V_l \ar[d]_{ \gamma_l}\\
U_i \ar[r]^{\beta_{ij}}&U_l}.
\end{equation}
Morphism of the quiver representations is given in the same way. This completes the proof of the lemma.
\end{proof}
Now consider the full subcategory of $\mathcal{K}_I$, objects of which are pairs $(M, \phi: I \otimes M \rightarrow M)$
satisfying the condition \eqref{ex}. Consider the image of this subcategory under the given equivalence.
Introduce the function $l(i,j)$ such that $I \otimes M_i=\bigoplus M_{l(i,j)}$. Then the maps $\alpha_{i,l(i,j)}$ act from $V_i \otimes M_i$ to $V_{l(i,j)} \otimes M_{l(i,j)}$.
\[
 \phi  \left(I  \otimes M\right)=\sum_{i=1}^n \phi \left(V_i \otimes I \otimes M_i \right)
 =\sum_{i,j} \left(\alpha(i,j)(V_i)\otimes M_{l(i,j)} \right).
\]
Therefore we have:
\[
 \phi  \circ (\id \otimes \phi)
 \left(\left(I \otimes I \right) \otimes M\right)=
 \phi (I \otimes \sum_{i,j} \left(\alpha(i,j)(V_i)\otimes M_{l(i,j)} \right))=\]
\[= \sum_{i,j,k} \left(\alpha(l(i,j),k)\alpha(i,j)(V_i)\otimes M_{l(l(i,j),k)} \right).
\]
\[I \otimes I \otimes M_i=\bigoplus_{j,k}V_i \otimes M_{l(l(i,j),k)}.\]
But $I \otimes I \otimes M_i=((I \odot I) \otimes M_i) \oplus
((I \wedge I) \otimes M_i). $ Thus, grouping the latter sum with respect to
the equivalence classes of modules $M_{l(l(i,j),k)}$:
\[V_i \otimes \bigoplus_{j,k} M_{l(l(i,j),k)}=V_i \otimes\bigoplus_m W_m \otimes M_m,\]
we have that every $W_m$ decomposes into a direct sum $W_m=W_m' \oplus W_m''$ so that
\[W_m' \otimes M_m \subset (I \odot I) \otimes M_i,\]
\[W_m'' \otimes M_m \subset (I \wedge I) \otimes M_i.\]

 The condition ~\eqref{ex} holds iff all the modules of the type
$V_i \otimes W_m'' \otimes M_m $ are mapped to $0$ under $\phi  \circ (\id \otimes \phi)$.
 In turn, this condition is equivalent to the linear dependencies system on $\alpha(l(i,j),k)\alpha(i,j)$
 (with $l(l(i,j),k)=m$), the number of which coincides with the dimension of $W_m''$.
 Therefore we obtain that the image of the considered subcategory is the category of finite dimension representations of the
 quiver $K_I$ with the set of relations of degree $2$. Using Lemma \ref{com}
 we obtain that this category is equivalent to the category of representations of the Lie algebra $L \rightthreetimes I$.

 \begin{remark}
 Remark \ref{com1} implies that if $\widehat{L}=L \rightthreetimes R$, where $L$ is a semisimple Lie algebra and the radical $R$
 is generated by the pre-image of $I=R/[R,R]$ under the factorization, isomorphic to $I$ as the $L$-module,
 then the category of representations of $\widehat{L}$ is equivalent to
 the category of representations of the quiver $K_I$ with some ideal of relations, which is contained in the ideal
 of relations for $L \rightthreetimes I$.
 \end{remark}

 \section{Representations of the quiver $K_I$ and wildness of Lie algebras with abelian radical}

 In the classification problem for the representations of $K_I$ we will find the finite-dimensional controlled wild subproblems.
 In particular if the quiver $K_I$ has the
 subquiver with the wild double (see, for example, \cite{Gab}), in particular, the wild subquiver
 without consecutive arrows, then the algebra $L \rightthreetimes I$ is wild.

Thus we easily derive the following result.
\begin{proposition} \label{sb}
 Let $I$ be an $L$-module such that there exists an
 $L$-module $M$ with the following property: $I \otimes M$ contains either 5 different indecomposable components
 or a component of the multiplicity $\geq 3$ or a component of the multiplicity 2 and another component.
 Then the Lie algebra $L \rightthreetimes I$ is controlled wild.
\end{proposition}

\begin{proof}
 In the considered cases double of $K_I$ has one of the following subquivers:

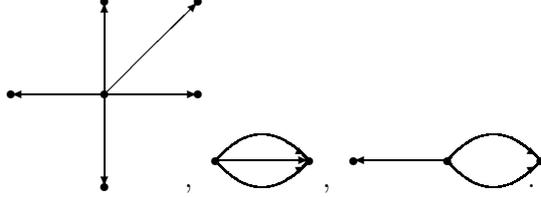
\begin{picture}(70.00,70.00)(-5,00)
 \put(35.00,35.00){\circle*{3.00}}
  \put(35.00,70.00){\circle*{3.00}}
 \put(35.00,0.00){\circle*{3.00}}
 \put(0.00,35.00){\circle*{3.00}}
  \put(70.00,35.00){\circle*{3.00}}
 \put(70.00,70.00){\circle*{3.00}}
 \put(35.00,35.00){\vector(1,0){35}}
  \put(35.00,35.00){\vector(-1,0){35}}
   \put(35.00,35.00){\vector(0,1){35}}
    \put(35.00,35.00){\vector(0,-1){35}}
     \put(35.00,35.00){\vector(1,1){35}}
\end{picture},
\begin{picture}(45.00,20.00)(-5,00)
 \put(0.00,10.00){\circle*{3.00}}
  \put(35.00,10.00){\circle*{3.00}}

 \put(34.00,8.00){\vector(2,1){0}}
  \put(34.00,12.00){\vector(2,-1){0}}
   \put(0.00,10.00){\vector(1,0){35}}
 \put(0.00,10.00){\bezier{200}(0,0)(17,20)(35,00)}
  \put(0.00,10.00){\bezier{200}(0,0)(17,-20)(35,00)}
\end{picture},
\begin{picture}(70.00,20.00)(-5,00)
 \put(0.00,10.00){\circle*{3.00}}
  \put(35.00,10.00){\circle*{3.00}}
\put(70.00,10.00){\circle*{3.00}}
 \put(69.00,8.00){\vector(2,1){0}}
  \put(69.00,12.00){\vector(2,-1){0}}
 \put(35.00,10.00){\vector(-1,0){35}}
 \put(35.00,10.00){\bezier{200}(0,0)(17,20)(35,00)}
  \put(35.00,10.00){\bezier{200}(0,0)(17,-20)(35,00)}
\end{picture}.

\end{proof}

\subsection{Lie algebras with a "big" abelian radical}

Now let us prove a stronger fact.
 Let $\Lambda_{N}$ be the highest weight of the module $N$, $\Lambda$ be the highest weight
 of the module $I$.
 \begin{definition}
The module $M$ will be called large for $I$ if $I \otimes M$ contains
 at least three indecomposable components and at least one of them is contained in the
 decomposition of the tensor product $I \otimes N$, $N  \neq M.$
  In this case we will say that I admits a large module.
\end{definition}

\begin{lemma} \label{qui}
Let $I$ be an indecomposable $L$-module and $I$ admits a large module $M$. Then the algebra $L \rightthreetimes I$ is wild.
\end{lemma}
\begin{proof}
 The case when $I \otimes M$ has nonisomorphic components is a partial case of Proposition \ref{sb}.
 Assume now that all these modules are pairwise nonequivalent. Denote them by $M_1$, $M_2$, $M_3$.
 Then without loss of generality we may assume that
 $\Lambda_{M_3}=\Lambda_{M}+\Lambda$. For $3 \leq i \leq 9$ denote by
 $M_i$ the module with the highest weight $\Lambda_{M_3}=\Lambda_{M}+(i-2)\Lambda$.
  Denote by $M_0$ a simple module $N$ such that $N \otimes I$ has the module $M_1$ as a direct summand and $N$ is not isomorphic to $M$
  (this module exists for the large module
  $\Lambda_{M}$). As the module of weight $\Lambda_{M_i}+2\Lambda$ is a direct summand of the module
  $(I \odot I) \otimes M_i$ we obtain that the classification problem for the representations of $K_I$ has the subproblem
  of classification of representations for the following quiver:

  
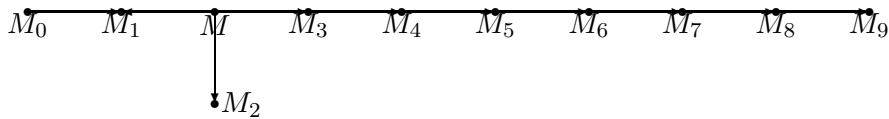
\begin{picture}(335.00,50.00)(-5,00)
 \put(0.00,45.00){\circle*{3.00}}
  \put(35.00,45.00){\circle*{3.00}}
\put(70.00,45.00){\circle*{3.00}}
\put(70.00,10.00){\circle*{3.00}}
 \put(105.00,45.00){\circle*{3.00}}
\put(140.00,45.00){\circle*{3.00}}
 \put(175.00,45.00){\circle*{3.00}}
\put(210.00,45.00){\circle*{3.00}}
\put(245.00,45.00){\circle*{3.00}}
\put(280.00,45.00){\circle*{3.00}}
\put(315.00,45.00){\circle*{3.00}}
 \put(00.00,45.00){\vector(1,0){35}}
 \put(70.00,45.00){\vector(-1,0){35}}
 \put(70.00,45.00){\vector(1,0){35}}
 \put(105.00,45.00){\vector(1,0){35}}
 \put(140.00,45.00){\vector(1,0){35}}
 \put(175.00,45.00){\vector(1,0){35}}
 \put(210.00,45.00){\vector(1,0){35}}
 \put(245.00,45.00){\vector(1,0){35}}
 \put(280.00,45.00){\vector(1,0){35}}
  \put(70.00,45.00){\vector(0,-1){35}}
  \put(00,45){\makebox(0,-10){$M_0$}}
  \put(35,45){\makebox(0,-10){$M_1$}}
  \put(70,45){\makebox(0,-10){$M$}}
  \put(105,45){\makebox(0,-10){$M_3$}}
  \put(140,45){\makebox(0,-10){$M_4$}}
  \put(175,45){\makebox(0,-10){$M_5$}}
  \put(210,45){\makebox(0,-10){$M_6$}}
  \put(245,45){\makebox(0,-10){$M_7$}}
  \put(280,45){\makebox(0,-10){$M_8$}}
  \put(315,45){\makebox(0,-10){$M_9$}}
  \put(75,10){\makebox(10,0){$M_2$}}
\end{picture}

  This quiver is wild (\cite{Naz}). If some of $M_i$ coincides with other one
then our quiver is wild, too.
\end{proof}

\begin{corollary}
 Let $\dim I \geq 3$. Then $L \rightthreetimes I$ is wild.
\end{corollary}
\begin{proof}
For any module $I$ there exist such a weight $\tilde{\lambda}$ that for any simple module $M$
with the highest weight greater then $\tilde{\lambda}$ the tensor product $I \otimes M$ has $\dim I$ irreducible direct summands.
(It can be easily obtained from the Costant formula and the formula of the dimensions, see for example \cite{dzh}).
Thus, if $\dim I \neq 3$ we can use Lemma \ref{qui}.
\end{proof}

 \subsection{The case of the two-dimensional module}
 Consider Lie algebras with two-dimensional abelian radicals. If it is decomposable then
 this quiver has the decomposition into a direct sum of two one-dimensional submodules.
 In this case, the quiver $K_I$ is a disjoint union of vertices with two loops.
 The module $I \wedge I$ is one-dimensional, therefore the classification problem of representations of some subquiver of this quiver
 is exactly the classification problem of pair of the matrices with one homogeneous condition of degree 2.
 The results of the paper  \cite{Sam} imply that this problem is always wild.
 Consider the case of the two-dimensional indecomposable module.
 Only Lie algebras of type
 $L= sl_2 \oplus \widehat{L}$ have the two-dimensional indecomposable module.
 $\widehat{L}$ is semisimple and acts on $I$ trivially.
 In this case $L\rightthreetimes I \backsimeq sl_2 \rightthreetimes I  \oplus \widehat{L}$.
 By Proposition \ref{tame}, the direct sum of a semisimple and a tame algebra is tame.
 Thus we may consider only the algebra $sl_2 \rightthreetimes I$.

  Consider the case of the two-dimensional module over $sl_2$. Then the quiver $K_I$ is of the following form:

 
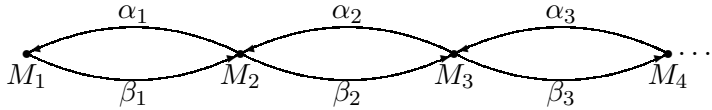
\begin{picture}(250.00,55.00)(-5,00)
 \put(00.00,20.00){\circle*{3.00}}
 \put(80.00,20.00){\circle*{3.00}}
 \put(160.00,20.00){\circle*{3.00}}
 \put(240.00,20.00){\circle*{3.00}}
 \put(0.00,20.00){\bezier{200}(0,0)(40,20)(80,00)}
 \put(0.00,20.00){\bezier{200}(0,0)(40,-20)(80,00)}
 \put(79.00,19.00){\vector(4,1){0}}
 \put(1.00,21.00){\vector(-4,-1){0}}
  \put(80.00,20.00){\bezier{200}(0,0)(40,20)(80,00)}
   \put(80.00,20.00){\bezier{200}(0,0)(40,-20)(80,00)}
 \put(159.00,19.00){\vector(4,1){0}}
 \put(81.00,21.00){\vector(-4,-1){0}}
  \put(160.00,20.00){\bezier{200}(0,0)(40,20)(80,00)}
   \put(160.00,20.00){\bezier{200}(0,0)(40,-20)(80,00)}
 \put(239.00,19.00){\vector(4,1){0}}
 \put(161.00,21.00){\vector(-4,-1){0}}

 \put(0,20){\makebox(0,-17){$M_{1}$}}
  \put(80,20){\makebox(0,-17){$M_{2}$}}
   \put(160,20){\makebox(0,-17){$M_{3}$}}
    \put(240,20){\makebox(0,-17){$M_{4}$}}
 \put(40,20){\makebox(0,30){$\alpha_{1}$}}
  \put(120,20){\makebox(0,30){$\alpha_{2}$}}
   \put(200,20){\makebox(0,30){$\alpha_{3}$}}
 \put(40,20){\makebox(0,-30){$\beta_{1}$}}
  \put(120,20){\makebox(0,-30){$\beta_{2}$}}
   \put(200,20){\makebox(0,-30){$\beta_{3}$}}
    \put(250,20){\makebox(0,0){$\ldots$}}
\end{picture}

 The module $I\wedge I$ is one-dimensional. Therefore we will have only the following conditions:
 $k_i\alpha_i \beta_i+l_i \beta_{i+1} \alpha_{i+1}=0$ with some constants $k_i$ and $l_i$.
 This problem is tame for all such constants. Thus the Lie algebra $sl_2\rightthreetimes I$
 is tame.
 The direct sums of tame and semisimple Lie algebras are tame, therefore all the algebras
 $L=(sl_2 \rightthreetimes I)\oplus \widehat{L}$  are tame.

 Now we can prove the following theorem.
 \begin{theorem}\label{abrad}
 Let $L=\widehat{L} \rightthreetimes I$ be an arbitrary Lie algebra with an abelian radical. Then $L$ is tame iff
  the module $I$ is one-dimensional or two-dimensional and indecomposable.
 \end{theorem}
 \begin{proof}
 If $\dim I >2$, then $L$ is wild by Lemma \ref{qui}.
 If the module $I$ is decomposable into a direct sum of two one-dimensional submodules, then $L$ is wild.
 Conversely if $I$ is the two-dimensional simple module then we have proved that $L$ is tame.
 If $I$ is one-dimensional then $L$ is tame by Remark \ref{rtame}.

 \end{proof}

 \section{The case of a nonabelian radical}

Now consider the algebras with nonabelian radicals. The algebra $L/[R,R]$ has an abelian radical.
 If this algebra is wild, then the Lie algebra $L$ is wild too.  Therefore we can consider only algebras with the tame
 quotient algebra by the square of the radical. Theorem \ref{abrad} gives a description of tame algebras with abelian
radicals.

\subsection{Lie algebras with two-dimensional indecomposable quotient algebra of the radical by its square}

 Consider now other algebras whose factor by the square of radical is isomorphic to $sl_2\rightthreetimes I\oplus L'$, $\dim I=2$
 or the algebras of the form $(sl_2\oplus L')\rightthreetimes R$, $R/[R,R]$ is two-dimensional. Let $a_+, a_-$  be some pre-images of the
 elements of $I$ with positive and negative weight such that their linear span is
 a module over $sl_2$. Let $N$ be the subalgebra generated by $a_+, a_-$.
Consider two following cases. $(i)$: $[N,N]=\lbrace 0 \rbrace$, $(ii)$: $[N,N]\neq\lbrace 0 \rbrace$.
 Consider the quotient algebra $L/[[R,R],[R,R]]$
and we will work with the algebras with abelian square of radical.

$(i)$ In this case $L=(sl_2 \rightthreetimes I \oplus L') \rightthreetimes M$.
Let $M$ be an irreducible module over $sl_2\rightthreetimes I$. Note that irreducible modules over $(sl_2 \rightthreetimes I \oplus L')$
are tensor products of irreducible modules over $sl_2 \rightthreetimes I$ and $L'$(see Proposition \ref{sumtame}).
If the direct summand $sl_2 \rightthreetimes I $ acts on $M$ trivially, then we have an algebra
with an abelian radical in the contradiction to our assumptions.
The quiver $K_I$ for the two-dimensional indecomposable module over $sl_2\oplus L'$ is decomposed into the disjoint union
of such quivers for the two-dimensional indecomposable module over $sl_2$. For
any finite subquiver of the considered quiver with considered relations it is easy to see that any irreducible representation
has a nonzero space only in one point. Therefore this algebra is isomorphic to
$sl_2 \rightthreetimes (I \oplus M')\oplus L'$ and is wild. Let
$M'$ be an arbitrary module and $M''$ be a maximal submodule of $M'$. Then we have that
$L=(sl_2 \rightthreetimes I) \rightthreetimes M'/M''$ is wild.
Therefore all considered algebras are wild.

$(ii)$
Let a Lie algebra $L$ be an extension of the Lie algebra $(L'\oplus sl_2) \rightthreetimes I$ with abelian kernel. If $R \neq N$ then
consider the Lie algebra $R / [N,N]$ ($[N,N]$ is one-dimensional ideal because $[[R,R],[R,R]]=0$). This is a Lie algebra from $(i)$.
Hence we can assume that the radical is generated by the irreducible module over $sl_2$. Then
$R=J \oplus [J,J]$ as a vector space. Thus we must consider the representations of the quiver
$K_I$ with some relations of degree $3$ lying in the ideal of relations for
the Lie algebra $L/[R,R]$.

 For the two-dimensional $sl_2$-module $I$ consider the following subquiver of $K_I$:

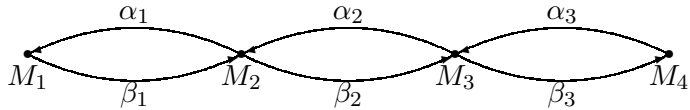
 \begin{picture}(250.00,55.00)(-5,00)
 \put(00.00,20.00){\circle*{3.00}}
 \put(80.00,20.00){\circle*{3.00}}
 \put(160.00,20.00){\circle*{3.00}}
 \put(240.00,20.00){\circle*{3.00}}
 \put(0.00,20.00){\bezier{200}(0,0)(40,20)(80,00)}
 \put(0.00,20.00){\bezier{200}(0,0)(40,-20)(80,00)}
 \put(79.00,19.00){\vector(4,1){0}}
 \put(1.00,21.00){\vector(-4,-1){0}}
  \put(80.00,20.00){\bezier{200}(0,0)(40,20)(80,00)}
   \put(80.00,20.00){\bezier{200}(0,0)(40,-20)(80,00)}
 \put(159.00,19.00){\vector(4,1){0}}
 \put(81.00,21.00){\vector(-4,-1){0}}
  \put(160.00,20.00){\bezier{200}(0,0)(40,20)(80,00)}
   \put(160.00,20.00){\bezier{200}(0,0)(40,-20)(80,00)}
 \put(239.00,19.00){\vector(4,1){0}}
 \put(161.00,21.00){\vector(-4,-1){0}}

 \put(0,20){\makebox(0,-17){$M_{1}$}}
  \put(80,20){\makebox(0,-17){$M_{2}$}}
   \put(160,20){\makebox(0,-17){$M_{3}$}}
    \put(240,20){\makebox(0,-17){$M_{4}$}}
 \put(40,20){\makebox(0,30){$\alpha_{1}$}}
  \put(120,20){\makebox(0,30){$\alpha_{2}$}}
   \put(200,20){\makebox(0,30){$\alpha_{3}$}}
 \put(40,20){\makebox(0,-30){$\beta_{1}$}}
  \put(120,20){\makebox(0,-30){$\beta_{2}$}}
   \put(200,20){\makebox(0,-30){$\beta_{3}$}}
\end{picture}

Put $\alpha_2=0$, let the maps $\alpha_1, \alpha_3$ be identity.
Relations for the algebra $L/[R,R]$ are of the form $k_i\alpha_i \beta_i+l_i \beta_{i+1}\alpha_{i+1}$, $k_i, l_i \in \mathbb{K}$
(because $(I \wedge I) \otimes M \simeq M$ for any indecomposable module $M$).
Therefore the condition $\beta_{i-1}\beta_{i}\beta_{i+1}=0$ does not lie in the ideal generated by these relations.
Thus we have only one relation of the form $k\beta_{i}\beta_{i+1} +l\beta_{i-1}\beta_{i}=0$,
 $k, l \in \mathbb{K}$. Hence this problem is wild.

  \subsection{Lie algebras with one-dimensional quotient algebra of radical by its square}

 Now consider the case when the quotient algebra of the radical $R$ by its square $[R,R]$ is one-dimensional.
 We can assume that a Lie algebra $L$ is unsolvable.
 We will prove the wildness of the quotients of this algebras, namely the algebras $L/[[R,R],[R,R]]$.
  Therefore we will assume that $[R,R]$ is abelian.
 Consider an extension of the algebra
 $L_0 \oplus I$ by the module $J$, where $L_0$ is a semisimple Lie algebra and $I$ is the one-dimensional algebra.
 This extension splits because the algebra $L_0 \oplus I$
 has trivial second cohomologies in any module (see for example \cite{Zus}).
 We can assume that $J$ is irreducible. Then Proposition \ref{tame} implies that this module is
 $L_0$-module and some element of $I$ acts on $J$ identically or trivially.
 In the latter case we obtain an algebra with the abelian radical.
  Therefore it remains only to consider the following case:
 $J$ is an irreducible $L_0$-module and there is an element of $I$ which acts on this module identically.
 Now let us prove the following proposition.
 \begin{proposition}\label{onedim}
 Let $L=(L_0 \oplus I) \rightthreetimes J$ be a Lie algebra such that $L_0$ is semisimple,
 $I$ is one-dimensional, $J$ is irreducible $L_0$-module, on which some element of $I$ acts identically.
 Then $L$ is wild.
 \end{proposition}
 \begin{proof}
 Consider $L$-modules $M$ of type $M=V_1 \otimes M_1 \oplus V_2 \otimes M_2$, where $M_1$ is trivial
 $L_0$-module, $M_2$ is isomorphic to $J$. Then $M$ has a structure of $L_0 \oplus I$-module and
 $L_0 \rightthreetimes J$-module. The structure of the $L_0 \oplus I$-module by Proposition
 \ref{tame} is given by two maps $\alpha_i:V_i \rightarrow V_i, i=1,2$
 ($im=\alpha(m)$, where $i$ is an element of $I$, $m$ is an element of $V_i$).
 Define a structure of $L_0 \rightthreetimes J$-module by a map $\beta:V_1 \rightarrow V_2$.
 Let $\mu$ be the multiplication map $J\otimes (V_1\otimes M_1)\rightarrow (V_2\otimes M_2)$.
 Then, using the equality $[i,j]=j$ for any $j \in J$, we obtain:
 \[\alpha_2 \mu(j,m_1)-\mu(j,\alpha_1 (m_1))= \mu(j,m_1),\]
 for any $m_1 \in M_1$. So, we have:
 \[(\alpha_2-id) \mu(j,m_1)=\mu(j,\alpha_1 (m_1)).\]
  \[(\alpha_2-id)\circ \mu(J,M_1)=\mu\circ (id,\alpha_1)(J,M_1).\]
  This statement is equivalent to the relation $(\alpha_2-id)\beta=\beta\alpha_1$.
  We obtain the following quiver with relations:

  
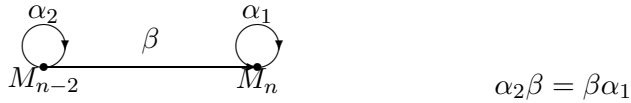
\begin{picture}(170.00,45.00)(-5,00)
 \put(00.00,10.00){\circle*{3.00}}
 \put(80.00,10.00){\circle*{3.00}}

 \put(0.00,10.00){\vector(1,0){80}}
 \put(88.00,16.00){\vector(0,-1){0}}
 \put(8.00,16.00){\vector(0,-1){0}}

 \put(80.00,18.00){\circle{16.00}}
 \put(00.00,18.00){\circle{16.00}}
 \put(0,10){\makebox(0,-10){$M_{n-2}$}}
  \put(80,10){\makebox(0,-10){$M_{n}$}}

 \put(40,10){\makebox(0,20){$\beta$}}

  \put(80,10){\makebox(0,40){$\alpha_1$}}
 \put(0,10){\makebox(0,40){$\alpha_2$}}
\end{picture}
$\alpha_2\beta=\beta\alpha_1$

  This problem is wild (See, for example, \cite{Han}).
 \end{proof}

 \begin{remark}
 The results of this paragraph and Theorem \ref{abrad} imply that any Lie algebra with nonabelian radical is wild.
 \end{remark}
 \subsection{The main theorem}
 Using the previous results we can prove the next theorem.
\begin{theorem}\label{main}
There exist only five classes of tame finite-dimensional Lie algebras over an algebraically closed field of zero characteristic:

1) a semisimple Lie algebras;

2) the one-dimensional Lie algebra;

3) a direct sums of semisimple semisimple Lie algebras and the one-dimensional algebra;

4) $sl_2 \rightthreetimes I$, where $I$ is the two-dimensional irreducible module;

5) direct sums of semisimple Lie algebras and $sl_2 \rightthreetimes I$.

Any other Lie algebra is controlled wild.
\end{theorem}
\begin{proof}
Any semisimple Lie algebra is tame by classical representation theory of the Lie algebras.
Any solvable Lie algebra is wild by Proposition \ref{solv}.

Let now $L$ be a Lie algebra without solvable direct summands.
Consider Levi decomposition of the given Lie algebra $\widehat{L}$: $\widehat{L}=L \rightthreetimes R$, where
 $L$ is the semisimple Lie algebra and $R$ is the radical. If $R/[R,R]$ is one-dimensional then by Propositions
 \ref{onedim} and \ref{tame} $L$ is wild iff $R$ is not one-dimensional.

 Assume now that $\dim I=\dim R/[R,R] > 1$. Then, if either the direct summand of
 $L$ nonisomorphic to $sl_2$ acts nontrivially on $I$ or more then one direct summand acts nontrivially on $I$,
 or $\dim I >2$, then $L/[R,R]$ is wild. Therefore $L$ is wild.
 If only $sl_2$ acts nontrivially on $I$, then either $R=I$ and $L$ is tame by Proposition \ref{sumtame},
  or $R\neq I$ and $L$ is wild by the results of the Section 4.2.

\end{proof}

 The author is grateful to Professor Yu. S. Samoilenko for suggesting the problem
 and for constant attention to this work.

\end{document}